\documentclass[a4paper,11pt]{amsart}

\usepackage{a4wide, amsmath, amsfonts, amssymb, mathrsfs, color, amsthm, enumitem}
\usepackage[hyperfootnotes=false]{hyperref} 

\numberwithin{equation}{section}

\newtheorem{theorem}{Theorem}[section]
\newtheorem{lemma}[theorem]{Lemma}
\newtheorem{corollary}[theorem]{Corollary}
\newtheorem{remark}[theorem]{Remark}
\newtheorem{proposition}[theorem]{Proposition}
\newtheorem{definition}[theorem]{Definition}

\newcommand{\dd}{\mathrm{d}}
\renewcommand{\d}{\,\mathrm{d}}
\newcommand{\var}{\text{-}\mathrm{var}}

\newcommand{\R}{\mathbb{R}}
\newcommand{\N}{\mathbb{N}}

\newcommand{\X}{{\bf X}}
\newcommand{\1}{\mathbf{1}}
\newcommand{\F}{\mathcal{F}}
\newcommand{\E}{\mathbb{E}}
\renewcommand{\P}{\mathbb{P}}

\title{Examples of It\^o c\`adl\`ag rough paths}

\author[Liu]{Chong Liu}
\address{Chong Liu, Eidgen\"ossische Technische Hochschule Z\"urich, Switzerland}
\email{chong.liu@math.ethz.ch}

\author[Pr\"omel]{David J. Pr\"omel}
\address{David J. Pr\"omel, University of Oxford, United Kingdom}
\email{proemel@maths.ox.ac.uk}

\date{\today}

\begin{document}

\begin{abstract}
  Based on a dyadic approximation of It\^o integrals, we show the existence of It\^o c\`adl\`ag rough paths above general semimartingales, suitable Gaussian processes and non-negative typical price paths. Furthermore, Lyons-Victoir extension theorem for c\`adl\`ag paths is presented, stating that every c\`adl\`ag path of finite $p$-variation can be lifted to a rough path.
\end{abstract}

\maketitle 



\section{Introduction}

Very recently, the notion of c\`adl\`ag rough paths was introduced by Friz and Shekhar \cite{Friz2017} (see also \cite{Chevyrev2017b,Chevyrev2017}) extending the well-known theory of continuous rough paths initiated by Lyons~\cite{Lyons1998}. These new developments significantly generalize an earlier work by Williams~\cite{Williams2001}. While \cite{Williams2001} already provides a pathwise meaning to stochastic differential equations driven by certain L\'evy processes, \cite{Friz2017,Chevyrev2017b} develop a more complete picture about c\`adl\`ag rough paths, including rough path integration, differential equations driven by c\`adl\`ag rough paths and the continuity of the corresponding solution maps. We refer to \cite{Lyons2007,Friz2010,Friz2014} for detailed introductions to classical rough path theory.

A c\`adl\`ag rough path is analogously defined to a continuous rough path using finite $p$-variation as required regularity, see Definition~\ref{def:rough path} and \ref{def:general rough path}, but (of course) dropping the assumption of continuity. Note that the notion of $p$-variation still works in the context of c\`adl\`ag paths without any modifications. Loosely speaking, for $p\in [2,3)$ a c\`adl\`ag rough path is a pair $(X,\mathbb{X})$ given by a c\`adl\`ag path $X\colon [0,T] \to\R^d$ of finite $p$-variation and its ``iterated integral'' 
\begin{equation}\label{eq:introduction}
  \mathbb{X}_{s,t} = ``\int_s^t (X_{r-}-X_s) \otimes \dd X_r \, ",\quad s,t\in [0,T],
\end{equation}
which satisfies Chen's relation and is of finite $p/2$-variation in the rough path sense. While the ``iterated integral'' can be easily defined for smooth paths~$X$ as for example via Young integration~\cite{Young1936}, it is a non-trivial question whether any paths of finite $p$-variation can be lifted (or enhanced) to a rough path. In the setting of continuous rough paths this question was answered affirmative by Lyons-Victoir extension theorem~\cite{Lyons2007a}. In Section~\ref{sec:extension theorem} we prove the analogous result in the context of c\`adl\`ag rough paths stating that every c\`adl\`ag path of finite $p$-variation for arbitrary non-integer $p \geq 1$ can be lifted to a rough path.

The theory of c\`adl\`ag rough paths provides a novel perspective to many questions in stochastic analysis involving stochastic processes with jumps, which play a very important role in probability theory. For a long list of successful applications of continuous rough path theory we refer to the book~\cite{Friz2014}. However, for applications of rough path theory in probability theory Lyons-Victoir extension theorem is not sufficient. Instead, it is of upmost importance to be able to lift stochastic processes to random rough paths via some type of stochastic integration. 

In Section~\ref{sec:ito rough paths} we focus on stochastic processes with sample paths of finite $p$-variation for $p\in (2,3)$, which is the most frequently used setting in probability theory, and construct the corresponding random rough paths using It\^o(-type) integration. More precisely, we define for a stochastic process~$X$ the ``iterated integral'' $\mathbb{X}$ (cf.~\eqref{eq:introduction}) as limit of approximating left-point Riemann sums, which corresponds to classical It\^o integration if $X$ is a semimartingale. The main difficulty is to show that $\mathbb{X}$ is of finite $p/2$-variation in the rough path sense. For this purpose we provide a deterministic criterion to verify the $p/2$-variation of $\mathbb{X}$ based on a dyadic approximation of the path and its iterated integral, see Theorem~\ref{thm:existence rough path}. As an application of Theorem~\ref{thm:existence rough path}, we provide the existence of It\^o c\`adl\`ag rough paths above general semimartingales (possibly perturbed by paths of finite $q$-variation), certain Gaussian processes and typical non-negative prices paths. Let us remark that related constructions of random c\`adl\`ag rough paths above stochastic processes are given in \cite{Friz2017} and \cite{Chevyrev2017b}, on which we comment in more detail in the specific subsections.  

\smallskip
\noindent{\bf Organization of the paper:} In Section~\ref{sec:extension theorem} the basic definitions and Lyons-Victoir extension theorem are presented. Section~\ref{sec:ito rough paths} provides the constructions of It\^o c\`adl\`ag rough paths.
\smallskip

\noindent{\bf Acknowledgment:} D.J.P. gratefully acknowledges financial support of the Swiss National Foundation under Grant No.~$200021\_163014$ and was affiliated to ETH Z\"urich when this project was commenced.

\section{C\`adl\`ag rough path and Lyons-Victoir extension theorem}\label{sec:extension theorem}

In this section we briefly recall the definitions of c\`adl\`ag rough path theory as very recently introduced in \cite{Friz2017,Chevyrev2017b} and present the Lyons-Victoir extension theorem in the c\`adl\`ag setting, see Proposition~\ref{prop:extension theorem}.
\smallskip

Let $D([0,T];E)$ be the space of c\`adl\`ag (right-continuous with left-limits) paths from $[0,T]$ into a metric space $(E,d)$. A partition~$\mathcal{P}$ of the interval~$[0,T]$ is a set of essentially disjoint intervals covering $[0,T]$, i.e. $\mathcal{P}=\{[t_i,t_{i+1}]\,:\, 0=t_0 < t_1<\cdots <t_n=T,\, n\in \mathbb{N}\}$. A path $X\in D([0,T];E)$ is of finite $p$-variation for $p\in (0,\infty)$ if 
\begin{equation*}
  \|X\|_{p\var} := \bigg(\sup_{\mathcal{P}} \sum_{[s,t]\in \mathcal{P}} d(X_s,X_t)^p \bigg)^{\frac{1}{p}}<\infty,
\end{equation*}
where the supremum is taken over all partitions $\mathcal{P}$ of the interval $[0,T]$ and the sum denotes the summation over all intervals $[s,t]\in \mathcal{P}$. The space of all c\`adl\`ag paths of finite $p$-variation is denoted by $D^{p\var}([0,T];E)$. For a two-parameter function $\mathbb{X}\colon \Delta_T \to \R^{d\times d}$ we define 
\begin{equation}\label{eq:p/2 variation}
  \|\mathbb{X}\|_{p/2\var} := \bigg(\sup_{\mathcal{P}} \sum_{[s,t]\in \mathcal{P}} |\mathbb{X}_{s,t}|^{\frac{p}{2}} \bigg)^{\frac{2}{p}},\quad p\in (0,\infty),
\end{equation}
where $\Delta_T :=\{ (s,t)\in [0,T]\,:\, s\leq t\}$ and $d\in\N$. Furthermore, we use the shortcut $X_{s,t}:= X_t-X_s$ for $X\in  D([0,T];\R^d)$.
\smallskip

For $p\in [2,3)$ the fundamental definition of a c\`adl\`ag rough path was introduced in \cite[Definition~12]{Friz2017} and reads as follows.

\begin{definition}\label{def:rough path}
  For $p \in [2, 3)$, a pair $\X = (X, \mathbb{X})$ is called \emph{c\`adl\`ag rough path} over $\R^d$  (in symbols $\X \in \mathcal{W}^p([0,T];\R^d)$) if $X\colon [0,T] \to \R^d$ and $\mathbb{X}\colon \Delta_T \to \R^{d\times d}$ satisfy:
  \begin{enumerate}[leftmargin=0.8cm]
    \item Chen's relation holds: $\mathbb{X}_{s,t} - \mathbb{X}_{s,u} - \mathbb{X}_{u,t} = X_{s,u} \otimes X_{u,t}$ for $0\leq s\leq u\leq t\leq T$.
    \item The map $ [0,T] \ni t \mapsto X_{0,t} + \mathbb{X}_{0,t}\in \R^d\times \R^{d\times d}$ is c\`adl\`ag.
    \item $\X = (X,\mathbb{X})$ is of finite $p$-variation in the rough path sense, i.e. $\|X\|_{p\var}+\|\mathbb{X}\|_{p/2\var}<\infty$.
  \end{enumerate} 
\end{definition}

An important subclass of rough paths are the so-called weakly geometric rough paths: For $N\geq 1$ let $G^N(\R^d) \subset T^N(\R^d) := \sum_{k=0}^N (\R^d)^{\otimes k}$ be the step-$N$ free nilpotent Lie group over~$\R^d$, embedded into the truncated tensor algebra $(T^N(\R^d),+,\otimes)$ which is equipped with the Carnot-Carath\'eodory norm $\|\cdot\|$ and the induced (left-invariant) metric~$d$. For more details we refer to \cite[Chapter~7]{Friz2010}. A rough path $\X= (X, \mathbb{X}) \in \mathcal{W}^p([0,T];\R^d)$ for $p\in [2,3)$ is said to be a weakly geometric rough path if $1 + X_{0,t} + \mathbb{X}_{0,t}$ takes values in $G^2(\R^d)$.
 
Note, while the constructions of rough paths carried out in Section~\ref{sec:ito rough paths} lead in general to non-geometric rough paths, it is always possible to recover a weakly geometric one.

\begin{remark}
  If $N=2$ and $p\in [2,3)$, one can easily verify that if $\X = (X, \mathbb{X})$ is a c\`adl\`ag rough path, then there exists a c\`adl\`ag function $F\colon [0,T] \to \R^{d\times d}$ of finite $p/2$-variation such that $1 + X_{0,t} + \mathbb{X}_{0,t} + F_t$ is a weakly geometric rough path, cf.~\cite[Exercise~2.14]{Friz2014}. 
\end{remark}

The notion of weakly geometric rough paths naturally extends to arbitrary low regularity $p\in [1,\infty)$, see \cite[Definition~2.2]{Chevyrev2017b}.

\begin{definition}\label{def:general rough path}
  Let $1\leq p < N+1$ and $N\in \N$. Any $\X \in D^{p\var}([0,T];G^N(\R^d))$ is called \emph{weakly geometric c\`adl\`ag rough path} over~$\R^d$.
\end{definition}

The next proposition is the Lyons-Victoir extension theorem (see in particular \cite[Corollary~19]{Lyons2007a}) in the context of c\`adl\`ag rough paths.

\begin{proposition}\label{prop:extension theorem}
  Let $p \in [1,\infty) \setminus \{2,3,\dots\}$ and $N\in \N$ be such that $ p < N+1$. For every c\`adl\`ag path $X\colon [0,T]\to \R^d$ of finite $p$-variation there exists a (in general non-unique) weakly geometric c\`adl\`ag rough path~$\X \in D^{p\var}([0,T];G^N(\R^d))$ such that $\pi_1(\X) = X$, where $\pi_1 \colon G^{N}(\R^d) \to \R^d$ is the canonical projection onto the first component.
\end{proposition}

\begin{proof}
  Let $X$ be a c\`adl\`ag $\R^d$-valued path of finite $p$-variation. By a slight modification of \cite[Theorem~3.1]{Chistyakov1998}, there exists a non-decreasing function $\varphi\colon [0,T] \to [0,\varphi(T)]$ with $\varphi(T)<\infty$ and a $1/p$-H\"older continuous function $g \colon [0,\varphi(T)]\to \R^d$ such that $X = g \circ \varphi$. Since $\varphi(t)$ is non-decreasing, the set $\mathcal{N}$ of discontinuity points of $\varphi$ is at most countable. Let us define a function~$\phi$ such that $\phi(t) = \varphi(t)$ for $t \in ([0,T]\setminus \mathcal{N}) \bigcup \{T\}$ and $\phi(t) = \varphi(t+) := \lim_{s \downarrow t, s \notin \mathcal{N}}\varphi(s)$ if $t \in \mathcal{N}$. It is easy to verify that $\phi$ is non-decreasing, c\`adl\`ag and $\phi(T) = \varphi(T)$. Moreover, since~$X$ is right-continuous and $g$ is continuous, we have $g \circ \phi = X$.  
  
  By \cite[Corollary~19]{Lyons2007a} there exits a weakly geometric $1/p$-H\"older continuous rough path~$\tilde{g}$ such that $\pi_1 (\tilde g)=g$. Now we define~$\tilde\X := \tilde{g} \circ \phi$. Since~$\phi$ is c\`adl\`ag and $\tilde{g}$ is continuous, $\tilde\X$ is also c\`adl\`ag. Furthermore, using \cite[Theorem~3.1]{Chistyakov1998} again we conclude that~$\tilde\X$ has finite $p$-variation and thus $\tilde\X\in D^{p\var}([0,T];G^{[p]}(\R^d))$ with $[p] := \max \{n\in \N : n \leq p\}$. Finally, it is obvious that $\pi_1(\tilde\X) = \pi_1(\tilde{g}) \circ \phi$ = $g \circ \phi = X$ and the extension of $\tilde\X$ to a weakly geometric c\`adl\`ag rough path $\X \in D^{p\var}([0,T];G^N(\R^d))$ for every $N\in\N$ with $ p < N+1$ is possible due to \cite[Theorem~20]{Friz2017}.
\end{proof}

Further conventions: The space $\R^d$ (resp. $\R^{d\times d}$) is equipped with the Euclidean norm~$|\cdot|$. For $X\in D([0,T];\R^d)$ the supremum norm is given by $\|X\|_\infty :=\sup_{t\in [0,T]} |X_t|$ and $X_{-}$ denotes the left-continuous version of $X$, i.e. $X_{-}(t):=X_{t-}:=\lim_{s\to t,\, s<t} X_s$ for $t\in (0,T]$ and $X_{-}(0):=X_{0-}:=X_0$. 
We write $A_{\vartheta}\lesssim B_{\vartheta}$ meaning that $A_{\vartheta}\leq CB_{\vartheta}$ for some constant $C>0$ independent of a generic parameter $\vartheta$ and  $A_{\vartheta}\lesssim_{\vartheta} B_{\vartheta}$ meaning that $A_{\vartheta}\leq C (\vartheta)B_{\vartheta}$ for some constant $C(\vartheta)>0$ depending on $\vartheta$. The indicator function of a set $A\subset \R$ or $A\subset D([0,T];\R^d)$ is denote by $\1_A$ and $x\wedge y := \min \{x,y\}$ for $x,y\in \R$.

\section{Construction of It\^o rough paths}\label{sec:ito rough paths}

In order to lift stochastic processes using It\^o type integration, we first prove a deterministic criterion to check the $p/2$-variation of the corresponding lift. The construction of random rough paths above (stochastic) processes is presented in the following subsections. 
\smallskip

For $X \in D([0,T];\R^d)$ or for (later) any c\`adl\`ag process~$X$, we define the dyadic (stopping) times $(\tau^n_k)_{n,k \in \N}$ by 
\begin{equation*}
  \tau^n_0 := 0\quad \text{and}\quad \tau^n_{k+1} := \inf\{ t \ge \tau^n_k \,:\, |X_t - X_{\tau^n_k}| \geq  2^{-n}\}.
\end{equation*}
Furthermore, for $t\in [0,T]$ and $n\in \N$ we introduce the dyadic approximation  
\begin{equation}\label{eq:approximation}
  X^n_t := \sum_{k=0}^\infty X_{\tau^n_k} \1_{(\tau^n_k, \tau^n_{k+1}]}(t)
  \quad \text{and}\quad 
  \int_0^t X^n_s \otimes \dd X_s := \sum_{k=0}^{\infty} X_{\tau_k^n}\otimes X_{\tau_k^n \wedge t, \tau_{k+1}^n\wedge t}. 
\end{equation}
Note that the integral $\int_0^t X^n_s \otimes \dd X_s$ is well-defined and $\|X^n - X_{-}\|_{\infty} \leq 2^{-n}$ for every $n\in\N$. 

\begin{theorem}\label{thm:existence rough path}
  Suppose that $X\in D^{p\var}([0,T];\R^d)$ for every $p > 2$ and there exist a function $\int_0^{\cdot} X_{-}\otimes \dd X \in D([0,T];\R^{d\times d})$ and a dense subset $D_T$ containing $T$ in $[0,T]$ satisfying that for every $t \in D_T$ and for every $\varepsilon \in(0,1)$, there exist an $N=N(t,\varepsilon) \in \N$ and a constant $c = c(p,\varepsilon)$ such that 
  \begin{equation}\label{eq:assumption}
    \left | \int_0^{t} X^n_s \otimes \dd X_s - \int_0^{t} X_{-}\otimes \dd X \right | \le c 2^{-n(1-\varepsilon)}\quad \text{for all }n \geq N. 
  \end{equation}  
  Setting for $(s,t) \in \Delta_T$
  \begin{equation*}
    \mathbb{X}_{s,t} := \int_s^t X_{r-}\otimes \dd X_r - X_{s}\otimes X_{s,t}:= \int_0^t X_{r-}\otimes \dd X_r-\int_0^s X_{r-}\otimes \dd X_r - X_{s}\otimes X_{s,t},
  \end{equation*}
  then $(X,\mathbb{X})\in \mathcal{W}^p([0,T];\R^d)$ for every $p \in (2,3)$.
\end{theorem}

To prove Theorem~\ref{thm:existence rough path}, we adapted some arguments used in the proof of \cite[Theorem~4.12]{Perkowski2016}, in which the existence of rough paths above typical continuous price paths is shown, cf. Subsection~\ref{subsec:typical price paths} below. As a preliminary step, we need a version of Young's maximal inequality (cf.~\cite{Young1936} or \cite[Theorem~1.16]{Lyons2007}) specific to the integral $\int X^n\otimes \dd X$. 

Recall that a function $c\colon \Delta_T \to [0,\infty)$ is called right-continuous super-additive if
\begin{equation*} 
  c(s,u) + c (u,t) \leq c(s,t)\quad \text{for}\quad 0\leq s \leq u \leq t \leq T,  
\end{equation*}
and $c(s,t)$ is right-continuous in $t$ for fixed $s$. Note that $X\in D^{p\var}([0,T];\R^d)$ if and only if there exits a right-continuous super-additive function $c$ s.t. $|X_{s,t}|^p \leq c(s,t)$ for all $(s,t) \in \Delta_T$.  

\begin{lemma}\label{lem:young estimate}
  Let $X\in D^{p\var}([0,T];\R^d)$ for every $p > 2$. Then it holds
  \begin{equation*}
    \bigg|\int_0^{t} X^n_r \otimes \dd X_r-\int_0^{s} X^n_r \otimes \dd   X_r-X_s \otimes X_{s,t}\bigg | 
    \lesssim \max\{2^{-n} c(s,t)^{1/q}, 2^{n (q-2)} c(s,t) + c(s,t)^{2/q}\},
  \end{equation*}
  for $q \in (2,3)$ and every super-additive function $c\colon \Delta_T\to [0,\infty)$ (which may depend on $q$) such that $|X_{s,t}|^q \leq c(s,t)$ for all $(s,t) \in \Delta_T$.
\end{lemma}

The proof follows the classical arguments used to derive Young's maximal inequality.

\begin{proof}
  Let $X \in D^{p\var}([0,T];\R^d)$ and let $X^n$ be its dyadic approximation as defined in~\eqref{eq:approximation}.
  
  1. If there exists no $k$ such that $\tau^n_k \in [s,t]$, then 
  \begin{equation*}
    \bigg|\int_0^{t} X^n_r \otimes \dd X_r-\int_0^{s} X^n_r \otimes \dd  X_r-X_s \otimes X_{s,t} \bigg | \lesssim 2^{-n} c(s,t)^{1/q}
  \end{equation*}
  due to the estimate $|X_{s,t}| \leq c(s,t)^{1/q}$. 
   
  2. If there exists a $k$ such that $\tau^n_k \in [s,t]$, we may assume that $s = \tau^n_{k_0}$ for some $k_0$. Otherwise, we just add $c(s,t)^{2/q}$ to the right-hand side. Let $\tau^n_{k_0} , \dots, \tau^n_{k_0 + N-1}$ be those $(\tau^n_k)_k$ which are in $[s,t)$. W.l.o.g. we may further suppose  that $N \geq 2$.  Abusing notation, we write $\tau^n_{k_0 + N} = t$. The idea is now to successively delete points $(\tau^n_{k_0 + \ell})$ from $\tau^n_{k_0} , \dots, \tau^n_{k_0 + N-1}$. Due to the super-additivity of $c$, there exist $\ell \in \{1, \dots, N-1\}$ such that
  \begin{equation*}
    c(\tau^n_{k_0 + \ell -1}, \tau^n_{k_0 + \ell + 1})
    \leq \frac{2}{N-1} c(s,t)
  \end{equation*}
  and thus 
  \begin{align*}
    |X_{\tau^n_{k_0 + \ell -1}}
    &  \otimes X_{\tau^n_{k_0 + \ell -1}, \tau^n_{k_0 + \ell}} + X_{\tau^n_{k_0 + \ell}}\otimes X_{\tau^n_{k_0 + \ell}, \tau^n_{k_0 + \ell + 1}} -   X_{\tau^n_{k_0 + \ell -1}} \otimes X_{\tau^n_{k_0 + \ell -1}, \tau^n_{k_0 + \ell + 1}}| \\
    &= |X_{\tau^n_{k_0 + \ell -1},\tau^n_{k_0 + \ell}}\otimes X_{\tau^n_{k_0 + \ell}, \tau^n_{k_0 + \ell + 1}}|
    \leq c(\tau^n_{k_0 + \ell -1},\tau^n_{k_0 + \ell + 1})^{2/q} 
    \leq \Big(\frac{2}{N-1} c(s,t)\Big)^{2/q}.
  \end{align*}
  Successively deleting in this manner all the points except $\tau^n_{k_0} = s$ and $\tau^n_{k_0 + N} = t$ from the partition generated by $\tau^n_{k_0} , \dots, \tau^n_{k_0 + N}$ leads to the estimate
  \begin{align*}
   \bigg|\int_0^{t} X^n_r \otimes \dd X_r-\int_0^{s} X^n_r  \otimes \dd X_r - X_s  \otimes X_{s,t}\bigg |
    &\leq \sum_{k=2}^{N} \Big(\frac{2}{k-1} c(s,t)\Big)^{2/q}  
    \lesssim N^{1-2/q} c(s,t)^{2/q}\\
    &\lesssim   (\#\{ k: \tau^n_k \in [s,t] \})^{1-2/q} c(s,t)^{2/q} + c(s,t)^{2/q}
  \end{align*}
  since $N\leq \#\{ k: \tau^n_k \in [s,t] \}$.

  Hence, 1. and 2., in combination with $\#\{ k: \tau^n_k \in [s,t] \} \lesssim 2^{nq} c(s,t)$, imply the assertion.
\end{proof}

With the auxiliary Lemma~\ref{lem:young estimate} at hand we come to the proof of Theorem~\ref{thm:existence rough path}. 

\begin{proof}[Proof of Theorem~\ref{thm:existence rough path}]
  It is straightforward to check that $(X,\mathbb{X})$ satisfies condition~(1) and~(2) of Definition~\ref{def:rough path} and $\|X\|_{p\var}<\infty$. Therefore, it remains to show the $p/2$-variation (in the sense of~\eqref{eq:p/2 variation}) of $\mathbb{X}$ for every $p > 2$. 
  
  Let $c$ be a right-continuous super-additive function with $|X_{s,t}|^q \leq c(s,t)$. Then for all $(s,t) \in \Delta_T \cap D_T^2$, using \eqref{eq:assumption} and Lemma~\ref{lem:young estimate}, for every $\varepsilon > 0$ and $q \in (2,3)$ we get a constant $c = c(p,q,\varepsilon)$ such that
  \begin{align}\label{eq:area variation}
    \begin{split}
    |\mathbb{X}_{s,t}| 
    & \le  c\Big(2^{-n(1-\varepsilon)} + \bigg|\int_0^{t} X^n_r \otimes \dd X_r-\int_0^{s} X^n_r \otimes \dd   X_r-X_s \otimes X_{s,t}\bigg |\Big) \\
    & \le c\Big( 2^{-n(1-\varepsilon)} + \max\{2^{-n} c(s,t)^{1/q},  2^{-n(2-q)}  c(s,t) + c(s,t)^{2/q}\} \Big),
    \end{split}
  \end{align}
  for all $n \geq N$, where $N \in \N$ may depend on $s,t$ and $\varepsilon$. 
  
  In the case that $c(s,t) \leq 1$, we set $\alpha := p/2$ for $p \in (2,3)$ and choose $n \geq N$ such that $2^{-n} \leq c(s,t)^{1/(\alpha(1-\varepsilon))}$. Taking this $n$ in~\eqref{eq:area variation}, we obtain 
  \begin{align*}
    |\mathbb{X}_{s,t}|^\alpha & \le c\Big( c(s,t) + \max\left\{c(s,t)^{1/(1-\varepsilon)} c(s,t)^{\alpha/q},  c(s,t)^{(2-q)/(1-\varepsilon) + \alpha}  + c(s,t)^{2\alpha/q}\right\} \Big) \\
    & = c\Big(c(s,t) + \max\left\{c(s,t)^{\frac{q+\alpha(1-\varepsilon)}{q(1-\varepsilon)}},  c(s,t)^{\frac{2-q+\alpha(1-\varepsilon)}{1-\varepsilon}}  + c(s,t)^{2\alpha/q}\right\}\Big)
  \end{align*}
  for some constant $c = c(\alpha,q,\varepsilon)$.
  Now we would like all the exponents in the maximum on the right-hand side to be larger or equal to 1. For the first term, this is satisfied as long as $\varepsilon<1$. For the third term, we need $\alpha\ge q/2$. For the second term, we need $\alpha\ge (q-1-\varepsilon)/(1-\varepsilon)$. Since $\varepsilon > 0$ can be chosen arbitrarily close to $0$, it suffices if $\alpha > q-1$. This means, choosing a $q_0 > 2$ close to $2$ enough such that $p/2 = \alpha > \max\{q_0/2, q_0-1\}$, we obtain that $|\mathbb{X}_{s,t}|^{p/2} \le c \cdot c(s,t)$ for some constant $c = c(p,q_0)$.
   
  For the remaining case $c(s,t) > 1$, we simply estimate
  \begin{equation*}
    |\mathbb{X}_{s,t}|^{p/2} \le c\Big( \Big\| \int_0^\cdot X_{r-} \otimes\dd X_r \Big\|_\infty^{p/2} + \|X\|_\infty^{p} \Big)\le c\bigg(\Big\| \int_0^\cdot X_{r-} \otimes \dd X_r \Big\|_\infty^{p/2} + \|X\|_\infty^{p}\bigg) c(s,t).
  \end{equation*}
  
  Therefore, $|\mathbb{X}_{s,t}|^{p/2} \le c\cdot c(s,t)$ for some constant $c= c(p)$ and for every $(s,t)\in \Delta_T \cap D_T^2$. Moreover, for an arbitrary $(s,t) \in \Delta_T$, picking any sequences $(s_k)_{k \in \N}$ and $(t_k)_{k \in \N}$ in $D_T$ such that $s_k \downarrow s$ and $t_k \downarrow t$ as $k \rightarrow \infty$, we have 
  $$|\mathbb{X}_{s,t}|^{p/2} = \lim_{k \rightarrow \infty}|\mathbb{X}_{s_k,t_k}|^{p/2} \le c(p) \limsup_{k \rightarrow \infty}c(s_k,t_k) \le c(p)\lim_{k \rightarrow \infty}c(s, t_k) = c(p) c(s,t),
  $$
  since $c(s,t)$ is right-continuous and super-additive. This ensures that $\|\mathbb{X}\|_{p\var} < \infty$.
\end{proof}

\begin{remark}
  All arguments in the proofs of Theorem~\ref{thm:existence rough path} and of Lemma~\ref{lem:young estimate} extend immediately from $\R^d$ to (infinite dimensional) Banach spaces. However, while the theory of continuous rough paths works for Banach spaces (cf.~\cite{Lyons1998,Lyons2007}), the current results about c\`adl\`ag rough paths are developed in finite dimensional settings (cf. \cite{Friz2017,Chevyrev2017b}). For this reason we also focus only on $\R^d$-valued paths and stochastic processes.     
\end{remark}

\subsection{Semimartingales}

Let $(\Omega, \mathcal{F},\mathbb{P})$ be a probability space with filtration $(\mathcal{F}_t)_{t\in [0,T]}$ satisfying the usual conditions. For a $\R^d$-valued semimartingale~$X$ we consider
\begin{equation}\label{eq:semimartingale lift}
  \mathbb{X}_{s,t} := \int_s^t (X_{r-} - X_s) \otimes\dd X_r = \int_0^t X_{r-} \otimes\dd X_r - \int_0^s X_{r-} \otimes\dd X_r - X_s\otimes  X_{s,t},\quad (s,t)\in \Delta_T,
\end{equation}
where the integration $\int X_{r-} \otimes\dd X$ is defined as It\^o integral. We refer to~\cite{Protter2004} and~\cite{Jacod2003} for more details on stochastic integration. 

\begin{proposition}\label{prop:semimartingales}
  Let $X$ be a $\R^d$-valued semimartingale. If $\mathbb{X}$ is defined as in~\eqref{eq:semimartingale lift} via It\^o integration, then $(X,\mathbb{X})\in \mathcal{W}^p([0,T];\R^d)$ for every $p\in (2,3)$ $\P$-almost surely.
\end{proposition}

\begin{proof}
  First note that every semimartingale possesses c\`adl\`ag sample paths of finite $p$-variation for any $p > 2$ (see e.g.~\cite[Chapter~II.1]{Protter2004} and~\cite{Lepingle1967}) and $\int X_{r-} \otimes\dd X$ has c\`adl\`ag sample paths (see e.g.~\cite[Theorem~I.4.31]{Jacod2003}). Therefore, in order to deduce Proposition~\ref{prop:semimartingales} from Theorem~\ref{thm:existence rough path}, it is suffices to verify that the condition~\eqref{eq:assumption} holds $\P$-almost surely for $\int X_{-} \otimes\dd X$ and its dyadic approximation $\int X^n \otimes\dd X $ defined via~\eqref{eq:approximation}.

  \textit{1.} Let us assume that $X = M$ is a square integrable martingale and denoted by $M^n$ its approximation defined as in~\eqref{eq:approximation}. By Burkholder-Davis-Gundy inequality we observe 
  \begin{equation}\label{eq:BDG inequalities}
    C(M,n):=\E\bigg [\left \|\int_0^{\cdot} M^n\otimes \dd M - \int_0^{\cdot} M_{-} \otimes\dd M \right \|^2_\infty \bigg] \lesssim 2^{-2n},\quad n\in \N,
  \end{equation}
  where the constant depends on the quadratic variation of $M$. Combining Chebyshev's inequality with~\eqref{eq:BDG inequalities}, we get 
  \begin{equation*}
    \P \bigg(\bigg\|\int_0^{\cdot} M^n\otimes \dd M - \int_0^{\cdot} M_{-} \otimes\dd M \bigg \|_\infty \geq  2^{-n(1-\epsilon)}\bigg) 
    \leq 2^{2n(1-\epsilon)}C(M,n)
    \lesssim 2^{-2n\epsilon}.
  \end{equation*}
  Since the right-hand side is summable in $n$, the Borel-Cantelli lemma gives 
  \begin{equation*}
    \bigg\|\int_0^{\cdot} M^n\otimes \dd M - \int_0^{\cdot} M_{-} \otimes\dd M \bigg \|_\infty 
    \lesssim_{\omega,\varepsilon} 2^{-n(1-\varepsilon)} \quad\P\text{-a.s.}
  \end{equation*}
 
  \textit{2.} Let $X = M$ be a locally square integrable martingale. Let $(\sigma_k)_{k \in \N}$ be a localizing sequence of stopping times for $M$ such that $\sigma_k \leq \sigma_{k+1}$, $\lim_{k \to \infty}\P(\sigma_k = T)= 1$, and for every~$k$, the stopped process $M^{\sigma_k}$ is a square integrable martingale. Thanks to 1. applied to every $M^{\sigma_k}$, for every $k$ there exists a $\Omega_k \subset \Omega$ with $\P(\Omega_k) = 1$ such that for all $\omega \in \Omega_k$, it holds that 
  \begin{equation*}
    \left\|\int_0^{\cdot \wedge \sigma_k} M^n\otimes \dd M - \int_0^{\cdot \wedge \sigma_k} M_{-}\otimes \dd M \right \|_\infty  \lesssim_{\omega,k} 2^{-n(1-\varepsilon)}
  \end{equation*}
  for any $n$. It follows immediately that \eqref{eq:assumption} holds for any $\omega \in \bigcup_{k \in \N} (\{\sigma_k = T\} \cap \Omega_k)$ and it holds that $\P(\bigcup_{k \in \N} (\{\sigma_k = T\} \cap \Omega_k)) = 1$.
  
  \textit{3.} By \cite[Theorem~III.29]{Protter2004}, every semimartingale $X$ can be decomposed as $X = X_0 + M + A$, where $X_0\in \R^d$, $M$ is a locally square integrable martingale and $A$ has finite variation. By 2. we obtain that $\left\|\int_0^{\cdot} X^n \otimes\dd M - \int_0^{\cdot} X_{-} \otimes\dd M \right \|_\infty  \lesssim_{\omega,\varepsilon} 2^{-n(1-\varepsilon)}$ $\P$-a.s.; on the other hand, since $\|X^n - X_{-}\|_{\infty} \leq 2^{-n}$, we also have $\left\|\int_0^{\cdot} X^n \otimes\dd A - \int_0^{\cdot} X_{-} \otimes\dd A \right \|_\infty \lesssim_{\omega} 2^{-n}$. 
\end{proof}

\begin{remark}\footnote{After completion of the present work, it was pointed out in \cite{Friz2017b} that also the It\^o lift of semimartingales can be constructed using an enhanced version of Burkholder-Davis-Gundy inequality.}
  Very recently, Chevyrev and Friz proved that every semimartingale can be lifted via the ``Marcus lift'' to a weakly geometric c\`adl\`ag rough path based on a new enhanced Burkholder-Davis-Gundy inequality, see~\cite[Section~4]{Chevyrev2017b}. Their result allows for deducing the existence of It\^o rough paths due to \cite[Proposition~16]{Friz2017}. However, let us emphasize that our approach directly provides the existence of an It\^o rough path only relying on classical It\^o integration and fairly elementary analysis (cf. Theorem~\ref{thm:existence rough path}). Moreover, it is independent of the results from \cite{Chevyrev2017b,Friz2017}. 
\end{remark}

Two natural generalizations of semimartingales are semimartingales perturbed by paths of finite $q$-variation for $q\in [1,2)$ and Dirichlet processes. While these stochastic processes are beyond the scope of classical It\^o integration, one can still construct corresponding random rough paths as limit of approximating Riemann sums.

For $Y \in D^{q\var}([0,T];\R^d)$ with $q \in [1,2)$, the Young integral  
\begin{equation*}
  \int_0^\cdot Y_{r-} \otimes \dd Y_r := \lim_{n\to \infty} \sum_{[s,t]\in \mathcal{P}^n} Y_{s-} \otimes Y_{s\wedge \cdot,t\wedge \cdot}
\end{equation*}
exists along suitable sequences of partition $(\mathcal{P}^n)_{n\in \N}$ and belongs to $D^{q\var}([0,T];\R^{d \times d})$, see for instance \cite{Young1936} or \cite[Proposition~14]{Friz2017}. In this case the Young integral can also obtained via the dyadic approximation $(Y^n)$ as defined in~\eqref{eq:approximation}. Indeed, using the Young-Loeve inequality (see e.g. \cite[Theorem~2]{Friz2017}) and a standard interpolation argument, one gets 
\begin{align*}
  \bigg\|\int_0^\cdot Y^n_r \otimes \dd Y_r -& \int_0^\cdot Y_{r-} \otimes \dd Y_r \bigg\|_{\infty} 
  \lesssim \|Y^n - Y_{-}\|_{q^\prime\var}\|Y\|_{q^\prime\var}\\
  &\lesssim \|Y^n - Y_{-}\|_{q\var}^{q/q^\prime}\|Y- Y^n\|_{\infty}^{1-q/q^\prime} \|Y\|_{q^\prime\var}
  \lesssim \|Y\|_{q\var}^{q/q^\prime+1}2^{-n(1-q/q^\prime)}
\end{align*}
for $1\leq q < q^\prime < 2$ and thus $\lim_{n \rightarrow \infty} \|\int_0^\cdot Y^n_r \otimes \dd Y_r - \int_0^\cdot Y_{r-} \otimes \dd Y_r \|_{\infty} = 0$.

As a consequence of Proposition~\ref{prop:semimartingales} and the previous discussion, it follows that semimartingale perturbed by paths of finite $q$-variation admit a natural rough path lift in the spirit of It\^o integration. A similar result for the canonical Marcus lift was presented in \cite[Section~5.1]{Chevyrev2017b}.

\begin{corollary}
  Let $Z=X+Y$ be semimartingale perturbed by paths of finite $q$-variation with $q\in [1,2)$, i.e. $X$ is a semimartingale and $Y$ is a stochastic process with sample paths of finite $q$-variation for $q\in [1,2)$. Then, there exists c\`adl\`ag rough path $(Z,\mathbb{Z})\in \mathcal{W}^p([0,T];\R^d)$ for every $p\in (2,3)$ $\P$-almost surely, where $\mathbb{Z}$ can be constructed as limit of approximating left-point Riemann sums. 
\end{corollary}

In the case $Z=X+Y$ for a stochastic process~$Y$ with continuous sample paths of finite $q$-variation with $q\in [1,2)$, the process $Z$ belongs to the class of so-called c\`adl\`ag Dirichlet processes, cf. \cite{Stricker1988} and \cite{Coquet2003}. Furthermore, let us remark that if $Y \in C^{0,2\var}([0,T];\R^d)$ admits a rough path lift, then it has to coincide with the Young integral, cf.~\cite[Exercise~2.12]{Friz2014}. Here $C^{0,2\var}([0,T];\R^d)$ denotes the closure of smooth paths on $[0,T]$ w.r.t. $|\cdot|_{2\var}$.

\subsection{Gaussian processes}

Let $(\Omega, \mathcal{F},\mathbb{P})$ be a probability space with filtration $(\mathcal{F}_t)_{t\in [0,T]}$ satisfying the usual conditions and let $X = (X^1,\dots,X^d)\colon  \Omega \times [0,T] \to \R^d$ be a $d$-dimensional Gaussian process. A natural candidate for the corresponding $\mathbb{X}=(\mathbb{X}^{i,j})_{i,j=1,\dots,d}$ is 
\begin{equation}\label{eq:gaussian lift}
  \mathbb{X}^{i,j}_{s,t} := \int_0^t X^i_{r-} \d X^j_r - \int_0^s X^i_{r-} \d X^j_r - X^i_{s} X^j_{s,t} 
  \quad \text{and} \quad 
  \mathbb{X}^{i,i}_{s,t} := \frac{1}{2}(X^i_{s,t})^2,\quad (s,t)\in \Delta_T,
\end{equation}
where $i\neq j$ and where the integral is given as $L^2$-limit of left-point Riemann-Stieltjes approximations. For more details on Gaussian processes in the context of rough path theory we refer to \cite[Chapter~15]{Friz2010}.

\begin{proposition}\label{prop:gaussian processes}
  Let $(X_t)_{t \in [0,T]}$ be a $d$-dimensional separable centered Gaussian process with independent components and c\`adl\`ag sample paths. If for every $q > 1$ 
  \begin{equation}\label{eq:variation condition for Gaussian process}
    \sup_{\mathcal{P},\mathcal{P}^\prime} \sum_{[s,t]\in \mathcal{P}, [u,v] \in \mathcal{P}^\prime} |\E[X_{s,t} \otimes X_{u,v}]|^{q} < \infty,
  \end{equation}
  then $(X, \mathbb{X})\in \mathcal{W}^{p}([0,T];\R^d)$ for every $p\in (2,3)$ $\P$-almost surely, where $\mathbb{X}$ is defined as in~\eqref{eq:gaussian lift} and $\mathbb{X}^{i,j}$ exists in the sense of an $L^2$-limit of Riemann-Stieltjes approximations for $i \neq j$.
\end{proposition}

\begin{proof}
  Proceeding as in \cite[Section~10.3]{Friz2017}, the sample paths of $X$ have finite $p$-variation for any $p > 2$ due to~\eqref{eq:variation condition for Gaussian process} and there exists a centered Gaussian process $\widetilde{X}$ with continuous sample paths such that $\widetilde{X} \circ F = X$, where $F^i(t):= \sup_{\mathcal{P}} \sum_{[u,v] \in \mathcal{P}} |X^i_u - X^i_v|_{L^2}^{2q}$ for every $i = 1,\dots,d$. 
   
  By \cite[Theorem~35~(iv)]{Friz2010c} the integral $\int_0^\cdot \widetilde{X}^i_r \d \widetilde{X}^j_r$ exists as the $L^2$-limit of Riemann-Stieltjes approximation and has continuous sample paths. Furthermore, using Young-Towghi's maximal inequality (see \cite[Theorem~3]{Friz2011c}) it can be verified that 
  \begin{equation}\label{Gaussian integral}
    \int_0^\cdot \widetilde{X}^i_r \d \widetilde{X}^j_r \circ F(t) = \lim_{|\mathcal{P}| \rightarrow 0}\sum_{[u,v] \in \mathcal{P}} X^i_{u}(X^j_{v\wedge \cdot} - X^j_{u\wedge \cdot}) = \lim_{|\mathcal{P}| \rightarrow 0}\sum_{[u,v] \in \mathcal{P}} X^i_{u-}(X^j_{v\wedge \cdot} - X^j_{u\wedge \cdot}),
  \end{equation}
  for $i,j=1,\dots,d$ with $i\neq j$, where the limits are taken in $L^2$ and in Refinement Riemann-Stieltjes sense (cf. \cite[Definition~1]{Friz2017}). We denote by $\int_0^{\cdot} X^i_{r-} \d X^j_{r}$ the integral from~\eqref{Gaussian integral}, which has c\`adl\`ag sample paths. 
  
  It remains check condition~\eqref{eq:assumption} for $\int_0^\cdot X^i_{r-} \d X^j_r$, which then implies the proposition by Theorem~\ref{thm:existence rough path}. With an abuse of notation, we write now $X$ for $X^i$ and $\widetilde{X}$ for $X^j$. Let $X^n$ be given as in~\eqref{eq:approximation} such that $\|X^n -  X_{-}\|_{\infty} \leq 2^{-n}$. We define for $Y^n := X^n - X_{-}$ and for $s,u \in [0,T]$ we set $R^{n}\binom{s}{u} := \E[Y^n_{0,s}Y^n_{0,u}]$ and $\widetilde{R}\binom{s}{u} := \E[\widetilde{X}_{0,s}\widetilde{X}_{0,u}]$.
  Thanks to \eqref{eq:variation condition for Gaussian process}, $\widetilde{R}$ has finite $q$-variation for any $q > 1$. We claim that $R^{n}$ has finite $p$-variation for any $p > 2$. Indeed, for every rectangle $[s,t] \times [u,v] \subset [0,T]^2$,  we have $|\E[Y^n_{s,t}Y^n_{u,v}]|^p \leq \|Y^n_{s,t}\|_{L^2}^p \|Y^n_{u,v}\|_{L^2}^p$. Using \cite[Proposition~1.7]{Basse2016} and the definition of $X^n$ we obtain that $\E[\|Y^n\|_{p\var}^p] \lesssim \E[\|X\|_{p\var}^p] < \infty$. Then, by Jensen's inequality we deduce that
  \begin{align*}  
    \E[\|Y^n\|_{p\var}^p] \geq \sup_{\mathcal{P}} \sum_{[s,t] \in \mathcal{P}}\E[|Y^n_{s,t}|^p] 
    \geq \sup_{\mathcal{P}} \sum_{[s,t] \in \mathcal{P}}(\E[|Y^n_{s,t}|^2])^{p/2} 
    = \sup_{\mathcal{P}}\sum_{[s,t] \in \mathcal{P}}\|Y^n_{s,t}\|_{L^2}^p,
  \end{align*}
  which means that $Y^n$ has finite $p$-variation w.r.t. the $L^2$-distance. Let $|||Y^n|||_{p\var}$ denote the $p$-variation norm of $Y^n$ in the $L^2$-distance, then $c_n(s,t):= |||Y^n|||_{p\var;[s,t]}^p$ is super-additive and $c_n(0,T) \lesssim \E[\|X\|_{p\var}^p]$ for all $n$. Hence, for any partitions $\mathcal{P}$, $\mathcal{P}^\prime$ of $[0,T]$ and for 
  \begin{equation*}
    R^{n}\binom{s,t}{u,v} := R^{n}\binom{s}{u} + R^{n}\binom{t}{v}- R^{n}\binom{s}{v}-R^{n}\binom{t}{u}, \quad u,v,s,t\in [0,T],
  \end{equation*}
  it holds that
  \begin{align*}
    \sum_{[s,t] \in \mathcal{P}, [u,v]\in \mathcal{P}^\prime}R^{n}\binom{s,t}{u,v}^p 
    \leq \sum_{[s,t] \in \mathcal{P}, [u,v]\in \mathcal{P}^\prime} c_n(s,t)c_n(u,v)
    \leq c_n(0,T)^2 \lesssim \E[\|X\|_{p\var}^p]^2.
  \end{align*}
  Now, for any $p > 2$, we can choose any $q > 1$ close to $1$ enough such that $1/p + 1/q > 1$. Since $Y^n$ and $\widetilde{X}$ are independent, applying Young-Towghi's maximal inequality  to the discrete integrals $\E[(\sum_{t_i \in \mathcal{P}} Y^n_{t_i}\widetilde{X}_{t_i,t_{i+1}})^2]$  and then sending $|\mathcal{P}|$ to zero, by Fatou's lemma we obtain that
  \begin{equation*}
    \E\bigg[\bigg(\int_0^t Y^n_{0,r} \d \widetilde{X}_r\bigg)^2\bigg] \lesssim V_p(R^n)V_q(\widetilde{R}), \quad t \in [0,T],
  \end{equation*}
  where $V_p$ denotes $p$-variation on $[0,T]^2$  in the sense of \cite[Definition~1]{Friz2011c}, given by
  \begin{equation*}
    V_p(R) := \sup_{\mathcal{P},\mathcal{P}^\prime}\bigg( \sum_{[s,t]\in \mathcal{P},[u,v]\in \mathcal{P}^\prime} R\binom{s,t}{u,v}^p \bigg)^{1/p}
  \end{equation*}
  for a function $R \colon [0,T]^2 \to \R$. By an interpolation argument we have for $p^\prime > p$, 
  \begin{equation*} 
    V_{p^\prime}(R^n) \leq V_p(R^n)^{p/p^\prime}\bigg(\sup_{s\neq t,u\neq v}\bigg|R^n\binom{s,t}{u,v}\bigg|\bigg)^{1-p/p^\prime}.  
  \end{equation*}
  Hence, noting that $|R^n\binom{s,t}{u,v}| = |\E[Y^n_{s,t}Y^n_{u,v}]| \lesssim 2^{-2n}$ due to $\|Y^n\|_\infty \leq 2^{-n}$, the above inequality applied for $p^\prime$ and $q$ with $1/p^\prime + 1/q > 1$ gives
  \begin{equation*}
    \E\bigg[\bigg(\int_0^t Y^n_{0,r} \dd \widetilde{X}_r\bigg)^2\bigg] \lesssim V_p(R^n)^{p/p^\prime}V_q(\widetilde{R})2^{-2n(1-p/p^\prime)}.
  \end{equation*}
  In particular, for a given $p>2$ and $\varepsilon > 0$, we choose $p^\prime = p/\varepsilon$, and a corresponding parameter $q$ close to $1$ enough such that $1/p^\prime + 1/q > 1$, we get
  \begin{equation*}
    \E\bigg[\bigg(\int_0^t Y^n_{0,r} \dd \widetilde{X}_r\bigg)^2\bigg] \lesssim_{\varepsilon} 2^{-2n(1-\varepsilon)}.
  \end{equation*}
  Then by Chebyshev's inequality, for each $n$ and each $t \in [0,T]$ we have (note that $Y^n_0 = 0$)
  \begin{equation*}
    \P\bigg(\bigg|\int_0^t Y^n_{r} \dd \widetilde{X}_r\bigg| \geq 2^{-n(1-2\varepsilon)}\bigg)
    \leq 2^{2n(1-2\varepsilon)}\E\bigg[\bigg(\int_0^t Y^n_{0,r} \dd \widetilde{X}_r\bigg)^2\bigg] 
    \lesssim_{\varepsilon} 2^{-2n\varepsilon}.
  \end{equation*}
  Since the right-hand side of the above inequality is summable over $n \in \N$, by the Borel-Cantelli lemma, we conclude that for every $t \in [0,T]$ there exists a $\Omega_t \subset \Omega$ with $\P[\Omega_t]=1$ such that for every $\omega \in \Omega_t$, when $n$ large enough ($n$ may depend on $\omega$ and $t$),
  \begin{equation*}
    \left|\int_0^t (X^n - X_{r-}) \dd \widetilde{X}_r\right| \leq 2^{-n(1-\varepsilon)}
  \end{equation*}
  holds. Let $D_T$ be any countable dense subset in $[0,T]$ containing $T$ and $\widetilde{\Omega}:= \bigcap_{t \in D_T} \Omega_t$. 
  Therefore, condition~\eqref{eq:assumption} is satisfied for every $\omega \in \widetilde{\Omega}$, which finishes the proof.
\end{proof}

\begin{remark}
  The Gaussian rough path as constructed in Proposition~\ref{prop:gaussian processes} is in fact a weakly geometric c\`adl\`ag rough path, which coincides with the one given in \cite[Theorem~60]{Friz2017}. However, while the proof of  \cite[Theorem~60]{Friz2017} is entirely based on time-change arguments and on corresponding well-known results for continuous Gaussian rough paths, the above proof gives a direct verification of the required rough path regularity via Theorem~\ref{thm:existence rough path}.
\end{remark}

\subsection{Typical price paths}\label{subsec:typical price paths} 

In recent years, initiated by Vovk, a model-free, hedging-based approach to mathematical finance emerged that uses arbitrage considerations to investigate which sample path properties are satisfied by ``typical price paths'', see for instance \cite{Vovk2008,Takeuchi2009,Perkowski2016}. In particular, Vovk's framework allows for setting up a model-free It\^o integration, see~\cite{Perkowski2016,Lochowski2016,Vovk2016}. Based on this integration, we show in the present subsection that ``typical price paths'' can be lifted to c\`adl\`ag rough paths.
\smallskip

Let $\Omega_+:=D([0,T];\R^d_+)$ be the space of all non-negative c\`adl\`ag functions $\omega\colon [0,T]\to\R^d_+$. The space $\Omega_+$ can be interpreted as all possible price trajectories on a financial market. For each $t\in [0,T]$, ${\mathcal{F}}_{t}^{\circ}$ is defined to be the smallest $\sigma$-algebra on~$\Omega_+$ that makes all functions $\omega\mapsto\omega(s)$, $s\in[0,t]$, measurable and ${\mathcal{F}}_{t}$ is defined to be the universal completion of ${\mathcal{F}}_{t}^{\circ}$. Stopping times $\tau\colon\Omega_+\to [0,T]\cup \{ \infty \} $ w.r.t. the filtration $({\mathcal{F}}_{t})_{t\in[0,T]}$ and the corresponding $\sigma$-algebras ${\mathcal{F}}_{\tau}$ are defined as usual. The coordinate process on $\Omega_+$ is denoted by $S=(S^1,\dots,S^d)$, i.e. $S_{t}(\omega):=\omega(t)$ and $S^i_{t}(\omega):=\omega^i(t)$ for $\omega=(\omega^1,\dots,\omega^d) \in \Omega_+$, $t\in [0,T]$ and $i=1,\dots,d$. 
\smallskip

A process $H\colon \Omega_+ \times [0,T]\to\R^d$ is a \emph{simple (trading) strategy} if there exist a sequence of stopping times $0 = \sigma_0 < \sigma_1 <  \sigma_2 < \dots$ such that for every $\omega\in\Omega_+$ there exist an $N(\omega)\in \N$ such that $\sigma_{n}(\omega)=\sigma_{n+1}(\omega)$ for all $n\geq N(\omega)$, and a sequence of $\F_{\sigma_n}$-measurable bounded functions $h_n\colon \Omega_+\to \R^d$, such that $H_t(\omega) = \sum_{n=0}^\infty h_n(\omega) \1_{(\sigma_n(\omega),\sigma_{n+1}(\omega)]}(t)$ for $t \in [0,T]$. Therefore, for a simple strategy $H$ the corresponding integral process
\begin{equation*}
  (H \cdot S)_t(\omega) :=  \sum_{n=0}^\infty h_n(\omega) S_{\sigma_n\wedge t, \sigma_{n+1} \wedge t}(\omega) 
\end{equation*}
is well-defined for all $(t,\omega) \in [0,T]\times \Omega_+$. For $\lambda > 0$ we write $\mathcal{H}_{\lambda}$ for the set of all simple strategies $H$ such that $(H\cdot S)_t(\omega) \geq - \lambda $ for all $(t,\omega) \in [0,T]\times \Omega_+$.

\begin{definition}
  \emph{Vovk's outer measure} $\overline{P}$ of a set $A \subset \Omega_+$ is defined as the minimal superhedging price for $\1_A$, that is
  \begin{equation*}
    \overline{P}(A) := \inf\Big\{\lambda > 0\,: \,\exists (H^n)_{n\in \N} \subset \mathcal{H}_{\lambda} \text{ s.t. } \forall \omega \in \Omega_+\,: \, \liminf_{n \to\infty} (\lambda + (H^n\cdot S)_T(\omega)) \ge \1_A(\omega)\ \Big\}.
  \end{equation*}
  A set~$\mathcal{A} \subset \Omega_+$ is called a \emph{null set} if it has outer measure zero. A property (P) holds for \emph{typical price paths} if the set ~$\mathcal{A}$ where (P) is violated is a null set.
\end{definition} 

Note that $\overline{P}$ is indeed an outer measure, which dominates all local martingale measures on the space $\Omega_+$, see \cite[Lemma~2.3 and Proposition~2.5]{Lochowski2016}. For more details about Vovk's outer measure we refer for example to \cite[Section~2]{Lochowski2016}.

\begin{proposition}\label{prop:typical price paths}
  Typical price paths belonging to $\Omega_+$ can be enhanced to c\`adl\`ag rough paths $(S,A)\in \mathcal{W}^p([0,T];\R^d)$ for every $p>2$ where
  \begin{equation*}
    A_{s,t} := \int_0^t S_{r-}\otimes \dd S_r -\int_0^s S_{r-}\otimes \dd S_r - S_s\otimes S_{s,t}, \quad (s,t)\in \Delta_T,
  \end{equation*}
  and $\int S_{-}\otimes \dd S $ denotes the model-free It\^o integral from \cite[Theorem~4.2]{Lochowski2016}.
\end{proposition}

\begin{proof}
  It follows from \cite[Theorem~1]{Vovk2011} that typical price paths belonging to $\Omega_+$ are of finite $p$-variation for every $p>2$. Hence, it remains to check condition~\eqref{eq:assumption} of Theorem~\ref{thm:existence rough path} to prove the assertion. 
  
  Let $S^n$ be the dyadic approximation of $S$ as defined in~\eqref{eq:approximation} for $n\in \N$  and let us recall that \cite[Corollary~4.9]{Lochowski2016} extends to the estimate 
  \begin{equation*}
    \overline{P} \bigg(\bigg \{\bigg \| \int_0^{\cdot} (S^n-S_{-})\otimes \dd S\bigg\|_\infty \geq a_n \bigg\} \cap \{|[S]_T|\leq b\} \cap \{\|S\|_\infty\leq b\} \bigg )
    \lesssim  6 (\sqrt{b} +2 +2b) \frac{c_n}{a_n}
  \end{equation*}
  where $c_n:= \| S^n - S\|_\infty \lesssim 2^{-n}$, $|[S]_T| := \big(\sum_{i,j=1}^d [S^i,S^j]_T^2\big)^{1/2}$ and $[S^i,S^j]$ denotes the quadratic co-variation as defined in \cite[Corollary~3.11]{Lochowski2016}. Due to the countable subadditivity of $\overline{P}$, it is enough to consider a fixed $b>0$. Setting $a_n:=2^{-(1-\varepsilon)n}$ for $\varepsilon \in (0,1)$ and applying the Borel-Cantelli lemma for $\overline{P}$ (see~\cite[Lemma~A.1]{Lochowski2016}), we get $\overline{P}(B_b)=0$ with
  \begin{align*}
    &B_b := \bigcap_{m\in \N} \bigcup_{n\geq m} A_{b,n} \quad \text{and}\quad\\
    &A_{b,n} :=  \bigg\{\bigg\| \int_0^{\cdot} (S^n-S_{-})\otimes \dd S\bigg\|_\infty \geq a_n\bigg \} \cap \{|[S]_T|\leq b\}\cap \{\|S\|_\infty\leq b\}.
  \end{align*}
  In particular, for typical price paths (belonging to $\Omega_+$) we have
  \begin{equation*}
    \left \| \int_0^{\cdot} (S^n-S_{-})\otimes \dd S \right\|_\infty \lesssim_\omega2^{-(1-\varepsilon)n}
  \end{equation*}
  for all $n\in \N$ and thus typical price paths satisfy condition~\eqref{eq:assumption}.
\end{proof}

Let us briefly comment on various aspects of Proposition~\ref{prop:typical price paths}.
\medskip

\begin{remark}~
  \begin{enumerate}[leftmargin=0.8cm]
    \item Proposition~\ref{prop:typical price paths} implies the (robust) existence of It\^o c\`adl\`ag rough paths in the sense that the set of all non-negative c\`adl\`ag paths which do not possess an It\^o rough path has measure zero with respect to all local martingale measures on $\Omega_+$. This justifies to take the existence of It\^o rough paths above price paths as an underlying assumption in model-free financial mathematics.  
    \item The non-existence of It\^o c\`adl\`ag rough paths above non-negative price paths leads to an pathwise arbitrage of the first kind, cf.~\cite[Proposition~2.6]{Lochowski2016}.
    \item In the case of continuous (price) paths the assertion of Proposition~\ref{prop:typical price paths} was obtained in~\cite[Theorem~4.12]{Perkowski2016}. 
    \item Proposition~\ref{prop:typical price paths} can be generalized in a straightforward manner from $\Omega_+$ to the more general sample spaces considered in~\cite{Lochowski2016}.
  \end{enumerate}
\end{remark}


\providecommand{\bysame}{\leavevmode\hbox to3em{\hrulefill}\thinspace}
\providecommand{\MR}{\relax\ifhmode\unskip\space\fi MR }
\providecommand{\MRhref}[2]{%
  \href{http://www.ams.org/mathscinet-getitem?mr=#1}{#2}
}
\providecommand{\href}[2]{#2}

\end{document}